\documentclass[11pt, a4paper]{amsart} 

\usepackage[margin=2.5cm]{geometry}

\usepackage[T1]{fontenc}
\usepackage{lmodern}

\usepackage{amssymb, amsrefs}
\usepackage{mathtools, mathrsfs, bm}
\usepackage{xcolor}
\usepackage{enumitem}
\setlist[enumerate,1]{label=(\arabic*)}

\definecolor{darkblue}{rgb}{0.0, 0.0, 0.55}
\usepackage{hyperref}
\hypersetup{
  colorlinks=true,
  linkcolor=darkblue,
  citecolor=darkblue,
  urlcolor=darkblue,
  linktoc=page,
}

\allowdisplaybreaks[1]
\numberwithin{equation}{section}

\newcommand{\C}{\mathbb{C}}
\newcommand{\R}{\mathbb{R}}

\newcommand{\ds}{\displaystyle}
\renewcommand{\le}{\leqslant}
\renewcommand{\ge}{\geqslant}
\renewcommand{\Re}{\operatorname{Re}}
\renewcommand{\Im}{\operatorname{Im}}

\DeclarePairedDelimiter{\ev}{\langle}{\rangle}
\DeclarePairedDelimiter{\abs}{\lvert}{\rvert}
\DeclarePairedDelimiter{\norm}{\lVert}{\rVert}
\providecommand\given{}
\DeclarePairedDelimiterX\set[1]{\{}{\}}{%
\renewcommand\given{\nonscript\:\delimsize\vert\nonscript\:\mathopen{}}%
#1
}

\theoremstyle{plain}
\newtheorem{theorem}{Theorem}[section]
\newtheorem{lemma}[theorem]{Lemma}
\newtheorem{proposition}[theorem]{Proposition}

\theoremstyle{definition}
\newtheorem{definition}[theorem]{Definition}

\theoremstyle{remark}
\newtheorem{remark}[theorem]{Remark}

\title[A Unified Integral Equation Approach to Conservation Laws for NLS]
{A Unified Integral Equation Approach to Conservation Laws for Nonlinear Schr\"odinger Equations}

\author{Shuji Machihara}
\address{Department of Mathematics, Faculty of Science, Saitama University, Saitama 338-8570, Japan}
\email{machihara@rimath.saitama-u.ac.jp}

\author[H. Miyazaki]{Hayato Miyazaki}
\address[]{Teacher Training Courses, Faculty of Education, Kagawa University, Takamatsu, Kagawa 760-8522, Japan}
\email{miyazaki.hayato@kagawa-u.ac.jp}

\author{Tohru Ozawa}
\address{Department of Applied Physics, Waseda University, Tokyo 169-8555, Japan}
\email{txozawa@waseda.jp}


\begin{document}

\begin{abstract}
We present a unified framework for the rigorous derivation of conservation laws
and related identities for nonlinear Schr\"odinger equations with power-type
nonlinearities.
This approach treats the equation in its Duhamel form and uses the
space-time integrability provided by Strichartz estimates, without relying on
smooth approximations or regularization procedures.
It was first
introduced by the third author in \cite{Ozawa06} and subsequently developed
in \cites{FM17,IMMO25}.

In this paper, we establish a single integral identity from which all of the
laws and identities considered here follow systematically. These include the
conservation of charge (mass), energy, and momentum, the pseudo-conformal
conservation law, and virial-type identities.
\end{abstract}

\maketitle

\section{Introduction}

In this paper, we consider the Cauchy problem for the nonlinear Schr\"odinger equation
\begin{align} \label{nls} \tag{NLS}
\begin{cases}
\ds i \partial_t u + \frac{1}{2}\Delta u = f(u), & (t,x) \in \R \times \R^d, \\
u(0,x) = \phi(x), & x \in \R^d,
\end{cases}
\end{align}
where $u \colon \R \times \R^d \to \C$ is the unknown function, $\Delta$ is the Laplacian on $\R^d$, and $f$ is the nonlinearity.
A typical example is the power-type nonlinearity $f(u)= \pm |u|^{p-1}u$. We adopt the standard assumptions on the nonlinearity $f$ used in \cites{Ginibre97, Kato89}:
\begin{enumerate}[label=(A\arabic*), ref=A\arabic*, itemsep=2mm]
\item \label{A1}
$f \in C^{1}(\C,\C)$, $f(0)=0$, and for some
\begin{align}
  &1<p<\infty,
  \qquad
  p \le 1+\frac{4}{d-2}\ \text{if } d\ge3,
  \label{p:con} \\[4pt]
&|f'(z)| \coloneqq
\max\left(
\left|\frac{\partial f}{\partial z}\right|,
\left|\frac{\partial f}{\partial \bar z}\right|
\right)
\le C(1+|z|^{p-1}) \notag
\end{align}
for all $z \in \C$.
\item \label{A2} $\Im(\bar z f(z))=0$ for all $z \in \C$.
\item \label{A3} There exists $V \in C^{1}(\C,\R)$ such that $V(0)=0$ and
$f(z)=\partial V/\partial \bar z$.
\end{enumerate}
Assumption \eqref{A1} is the analytic condition ensuring local well-posedness and the finiteness of the relevant space-time quantities.
Assumption \eqref{A2} is the algebraic condition underlying the conservation of charge.
Combined with \eqref{A3}, it implies the gauge invariance $f(e^{i\theta}z)=e^{i\theta}f(z)$ for all $z \in \C$ and $\theta \in \R$.
Assumption \eqref{A3} provides the Hamiltonian structure underlying the conservation of energy and momentum.
Under \eqref{A2} and \eqref{A3}, one has $V(z)=V(|z|)$ for all $z \in \C$, and we write
\[
V'(|z|)=\left.\frac{dV}{dr}(r)\right|_{r=|z|}.
\]

The Cauchy problem for \eqref{nls} has been extensively studied; see, for instance, \cites{Cazenave, GV79, Kato87, SS99}.
A standard approach consists of two steps.
One first proves local well-posedness for the associated integral equation by a contraction argument based on Strichartz estimates, and then extends the local solution by means of a priori bounds provided by conservation laws.

Throughout this paper, we denote the standard inner product in $L^{2}(\R^d)$ by
\[
  (f,g)\coloneqq \int_{\R^d} f(x)\overline{g(x)}\,dx
\]
for $f,g \in L^{2}(\R^d)$.
Formally, solutions to \eqref{nls} satisfy the conservation of charge (mass), energy, and momentum:
\[
  I(u) \coloneqq \norm{u}_{L^{2}}^{2}, \qquad
  E(u) \coloneqq \frac12 \norm{\nabla u}_{L^{2}}^{2} + \int_{\R^d} V(u)\,dx, \qquad
  P(u) \coloneqq \Im \int_{\R^d} \overline{u}\,\nabla u\,dx.
\]
Among these quantities, $I(u)$ and $E(u)$ are fundamental for obtaining
the a priori bounds needed for global well-posedness in $L^{2}$ and $H^{1}$,
respectively.
The momentum, the pseudo-conformal conservation law, and virial identities also play important roles in the qualitative analysis of solutions (e.g., \cites{Glassey77,MR05}).

Under assumption \eqref{A2}, the conservation of charge is formally obtained by taking the real part of the scalar product of \eqref{nls} with $iu$:
\begin{align*}
  0 = 2\Re\left(i\partial_t u+\frac12\Delta u-f(u), iu\right)
  =\frac{d}{dt}\norm{u}_{L^{2}}^{2}.
\end{align*}
Under \eqref{A3}, the conservation of energy is formally derived by taking the real part of the scalar product of \eqref{nls} with $-\partial_t u$:
\begin{align*}
  0 = 2\Re\left(i\partial_t u+\frac12\Delta u-f(u), -\partial_t u\right)
  =\frac{d}{dt}E(u).
\end{align*}
Similarly, the conservation of momentum is formally obtained by pairing \eqref{nls} with $\nabla u$:
\begin{align*}
  0 = 2\Re\left(i\partial_t u+\frac12\Delta u-f(u), \nabla u\right)
  = 2\Re(i\partial_t u,\nabla u)-2\Re(f(u),\nabla u)
  = \frac{d}{dt}P(u(t)).
\end{align*}
Here the nonlinear term vanishes formally because \eqref{A3} implies
\[
2\Re(f(u)\nabla\overline{u})=\nabla(V(u)),
\]
and the integral of $\nabla(V(u))$ over $\R^d$ vanishes by the divergence theorem.

The solution also formally satisfies the pseudo-conformal conservation law
\begin{align}
  \frac12\norm{J(t)u(t)}_{L^{2}}^{2}
  +t^{2}\int_{\R^d}V(u(t))\,dx
  =
  \frac12\norm{x\phi}_{L^{2}}^{2}
  +\int_{0}^{t}s\left(\int_{\R^d}W(u(s))\,dx\right)ds,
  \label{eq:pc} \tag{PC}
\end{align}
where $J(t)=x+it\nabla$ and $W(u)=(d+2)V(u)-\frac{d}{2}V'(u)|u|$.
The operator $J(t)$ is the generator of the Galilei transform and is connected with the Galilean symmetry of equation \eqref{nls}.
Formally, this identity is obtained by pairing \eqref{nls} with $iJ(t)^2u$.
Since the expansion of $J(t)^2u$ contains the scaling (dilation) generator
$x\cdot \nabla u + du/2$, the nonlinear term can be rewritten in terms
of the potential energy and its scaling derivative:
\begin{align*}
  \Re(f(u), i J(t)^{2}u) &=  -t \frac{d}{d \lambda} \left. \int_{\R^{d}} V \left( \lambda^{d/2} u(\lambda x)\right) dx \right|_{\lambda=1}
  -t^{2} \Im (f(u) , \Delta u) \\
  &= -t \left( \frac{d}{2} \int_{\R^{d}} V'(|u|)|u| \, dx - d \int_{\R^{d}} V(u) \, dx  \right) - t^{2} \frac{d}{dt} \int_{\R^{d}} V(u) \, dx.
\end{align*}
For the details, see Section \ref{sec:auxiliary_identities} below.
On the other hand, the operator $i\partial_t+\frac12\Delta$ commutes with $J(t)$.
Therefore, \eqref{eq:pc} is formally obtained from
\begin{align*}
0 &= 2\Re \left( i \partial_{t} u + \frac{1}{2}\Delta u - f(u), iJ(t)^{2}u \right) \\
&= \frac{d}{dt}\norm{J(t)u}_{L^{2}}^2
+2t \left( \frac{d}{2} \int_{\R^{d}} V'(|u|)|u| \, dx - d \int_{\R^{d}} V(u) \, dx  \right) +2 t^{2} \frac{d}{dt} \int_{\R^{d}} V(u) \, dx.
\end{align*}

For the virial identity
\begin{align}
    \norm{xu(t)}_{L^2}^2 = \norm{x\phi}_{L^2}^2
    + 2t \Im (\nabla u(t),xu(t)) + 2t^2 E(\phi)
    - 2 \int_0^t \int_0^s \int_{\R^n} W(u(\tau))\, dx d\tau ds,
    \label{virial}
    \tag{V}
\end{align}
pairing \eqref{nls} with $i|x|^{2}u$ and taking the real part, we obtain
\begin{align}
  0 &=
  2\Re\left(i\partial_t u+\frac12\Delta u-f(u),\, i|x|^{2}u\right)
  =
  \frac{d}{dt}\norm{xu(t)}_{L^{2}}^{2}
  -2\Im(\nabla u(t),xu(t)).
  \label{eq:vi1}
\end{align}
Similarly, pairing \eqref{nls} with $x\cdot\nabla u+du/2$
and taking the real part yields
\begin{align}
\begin{aligned}
  0 &= 2 \Re\left(i\partial_t u+\frac12\Delta u-f(u), x\cdot\nabla u+ \frac{d}{2}u \right) \\
  &= -\frac{d}{dt}\Im(xu(t),\nabla u(t))
    +\norm{\nabla u(t)}_{L^{2}}^{2}
    -\int_{\R^d}W(u(t))\,dx.
\end{aligned}
  \label{eq:vi2}
\end{align}
Combining \eqref{eq:vi1} and \eqref{eq:vi2} with the conservation of energy,
we formally obtain \eqref{virial}.

\eqref{eq:pc} and \eqref{virial}
are closely related to the Galilean and scaling symmetries of the Schr\"odinger equation.
The Hamiltonian structure in \eqref{A3} allows the nonlinear terms to be
expressed through the real-valued quantity $W(u)$.
In the $L^2$-critical case $f(u)= \pm |u|^{4/d}u$, one has
$W(u)\equiv 0$, so that \eqref{eq:pc} becomes an exact
conservation law.

However, the above computations are only formal.
To justify them rigorously, we need sufficiently smooth and sufficiently integrable solutions.
For instance, the derivation of the energy and momentum conservations requires
$H^{2}$-regularity, whereas the derivation of \eqref{eq:pc} and the virial-type identities require $H^{2} \cap \mathcal{F}H^{2}$-regularity.
Thus, for low-regularity solutions such as $L^{2}$-, $H^{s}$-, $H^{1}$-, or
$\Sigma$-solutions, where $\Sigma \coloneqq H^{1} \cap \mathcal{F}H^{1}$, these
formal computations are not directly meaningful, and additional arguments are
needed to justify these laws and identities.

There are two standard ways to justify the above formal computations.

\begin{enumerate}
\item \textbf{Approximation by smooth solutions.}
Using continuous dependence on the initial data, one approximates the given
solution by sufficiently smooth solutions (cf. \cites{Kato87, Kato89}).

\item \textbf{Regularization of the equation.}
One introduces a regularized equation
(for instance by Yosida approximation
\cites{GV79, Tsutsumi87})
or a regularizing spatial weight such as $e^{-\varepsilon|x|^{2}}$
to recover spatial integrability (cf.
\cite[\S6.5]{Cazenave}).
\end{enumerate}
Both approaches rely on approximation procedures followed by delicate
limiting arguments.

An observation due to Ginibre is also relevant here.
In the $H^{1}$-supercritical case, compactness methods yield
global weak solutions satisfying the energy inequality
\[
E(u(t)) \le E(u_0).
\]
However, neither uniqueness nor energy conservation is known in that
setting.

\textit{Uniqueness would imply energy conservation by the previous
inequality and the time-reversal invariance of the equation}
\cite[p.~99]{Ginibre97}.

This remark suggests that the justification of conservation laws is
closely related to the uniqueness theory for the corresponding
solution class.
This also motivates the perspective adopted in this paper.

In this paper, we revisit an alternative approach initiated by the third
author in \cite{Ozawa06} and further developed in \cites{FM17, IMMO25}.
This approach derives the above conservation laws and identities directly from the integral equation associated with \eqref{nls}, without introducing approximation sequences or auxiliary regularization.
The key idea is to exploit the space-time integrability provided by Strichartz estimates and to work entirely at the level of the integral equation.
We also note that related integral-equation arguments already appeared in
earlier work on nonlinear Schr\"odinger equations with singular interactions
\cite{AT01}, and have more recently been used in
the derivation of
weak formulations for certain quantum hydrodynamic models
\cites{A21,AMS24,S25}.
These examples suggest that such arguments are relevant beyond the present setting.

More precisely, in Proposition \ref{prop:main}, we establish an
integral identity for the inner product involving solutions to the integral equation.
This identity may be viewed as a rigorous integral-equation counterpart to
the formal pairing identities at the differential-equation level.
We refer to this identity as the master identity.
From this identity, we systematically derive
\begin{itemize}
\item conservation of charge,
\item conservation of energy,
\item conservation of momentum,
\item pseudo-conformal conservation law,
\item virial-type identities.
\end{itemize}
Thus these laws and identities all follow from the master identity.

This paper is organized as follows.
Section \ref{sec:prelim} recalls the notion of solutions and Strichartz estimates, and develops several auxiliary identities and lemmas, most of which are independent of equation \eqref{nls}.
Section \ref{sec:unified_identity} presents the master identity and uses it to prove the conservation of charge, energy, and momentum, as well as the pseudo-conformal conservation law and the virial-type identities.

\subsection*{Notation}
We write $L^p = L^p(\R^d)$ for the Lebesgue space on $\R^{d}$, and $L^\rho(I; X)$
for the Bochner space on an interval $I$ with values in a Banach space $X$.
$\mathcal{F}[u] = \widehat{u}$ is the usual Fourier transform of a function $u$ on $\R^{d}$, and $\mathcal{F}^{-1}[u] = \check{u}$ is its inverse.
For $s \in \R$ and $r \in (1, \infty)$, we denote the Bessel potential space on $\R^{d}$ by
\begin{align*}
H^{s, r} &= H^{s, r}(\R^{d}) \coloneqq \set*{ u \in \mathscr{S}'(\mathbb{R}^{d}) \given \norm{u}_{H^{s,r}} \coloneqq \norm{(1- \Delta)^{\frac{s}{2}} u}_{L^r} < \infty },
\end{align*}
where $\mathscr{S}'(\R^{d})$ stands for the space of tempered distributions on $\R^{d}$, and
\[
  (1- \Delta)^{\frac{s}{2}} = \mathcal{F}^{-1} \langle \xi \rangle^{s} \mathcal{F}, \qquad
  \langle \xi \rangle \coloneqq (1+|\xi|^2)^{1/2}
\]
is the Bessel potential operator. We denote $H^{s} = H^{s, 2}$ for $s \in \R$; in particular, $H^{0} = L^{2}$.
The weighted Sobolev space is also defined by
$\mathcal{F}H^{s} = \mathcal{F}H^{s}(\R^{d}) \coloneqq \{ u \in L^{2}(\mathbb{R}^d)
\mid \norm{u}_{\mathcal{F}H^{s}} \coloneqq \norm{\ev{x}^{s}u}_{L^2} < \infty \}$ for $s \ge 0$.
We set $\Sigma = H^{1} \cap \mathcal{F} H^{1}$.

For a Banach space $X$ and its dual $X'$, the duality pairing is denoted by $\ev{\cdot, \cdot}$. More precisely, if $X = H^{s,r'}(\R^{d})$ with $s \ge 0$ and $1 < r < \infty$, then $X' = H^{-s, r}(\R^{d})$ and the pairing $\ev{f, g}$ for $f \in X$ and $g \in X'$ (or vice versa) is defined by
\begin{align*}
\ev{f, g}
&= \int_{\R^{d}} (1- \Delta)^{\frac{s}{2}} f(x) \cdot (1- \Delta)^{-\frac{s}{2}} g(x)\, dx.
\end{align*}
In particular, if $r=2$, then $X = H^{s}(\R^{d})$, $X' = H^{-s}(\R^{d})$, and $\ev{f, g} = ( \hat{f}, \overline{\hat{g}})$, where $(\cdot, \cdot)$ is the inner product in $L^{2}(\R^{d})$.

For two Banach spaces $X$ and $Y$, we denote by $X + Y$ the Banach space consisting of all elements $f$ such that $f = f_{1} + f_{2}$ with $f_{1} \in X$ and $f_{2} \in Y$, equipped with the norm
$\norm{f}_{X + Y} = \inf_{f = f_{1} + f_{2}} \left( \norm{f_{1}}_{X} + \norm{f_{2}}_{Y} \right)$.
The space $X \cap Y$ is equipped with the norm $\norm{f}_{X \cap Y} = \norm{f}_{X} + \norm{f}_{Y}$.
Since $(X+Y)' = X' \cap Y'$, the duality pairing on $(X' \cap Y') \times (X+Y)$ is defined by
\[
    \ev{f, g}_{(X' \cap Y') \times (X+Y)}
    = \ev{f, g_{1}}_{X' \times X} + \ev{f, g_{2}}_{Y' \times Y}
\]
for $f \in X' \cap Y'$ and $g = g_{1} + g_{2} \in X + Y$.
Remark that $\ev{f, g}_{(X' \cap Y') \times (X+Y)}$ is independent of the decomposition $g = g_{1} + g_{2}$.

A pair of exponents $(\rho, \gamma)$ is said to be admissible if
\begin{align*}
\rho, \gamma \in [2, \infty],\quad
\frac{2}{\rho} &= d\left(\frac12-\frac{1}{\gamma} \right), \quad (d,\rho, \gamma) \neq (2,2,\infty).
\end{align*}

\section{Preliminaries} \label{sec:prelim}

\subsection{Definition of solutions}

Let $U(t)=e^{it\Delta/2}$ denote the free Schr\"odinger group.
Motivated by the Duhamel formula, we define solutions to \eqref{nls}.

\begin{definition}[Solution]
Let $s \ge 0$.
Let $I \ni 0$ be an interval, and fix $t_{0} \in I$.
We say that a function $u \colon I \times \R^d \to \C$ is an $H^{s}$-solution
to \eqref{nls} on $I$ if there exists an admissible pair $(\rho, \gamma)$ such that
\[
  u \in C(I;H^{s}(\R^d)) \cap L^{\rho}(I; H^{s,\gamma}(\R^d)),
\]
and
\begin{align} \label{eq:01}
u(t)
= U(t-t_{0})u(t_{0}) - i \int_{t_{0}}^{t} U(t-s)f(u(s))\,ds
\end{align}
holds in $H^{s}(\R^d)$ for every $t \in I$.
Moreover, $u$ is called a $\Sigma$-solution if $u$ is an $H^{1}$-solution and satisfies
\[
  J(\cdot)u \in C(I;L^{2}(\R^d)) \cap L^{\rho}(I;L^{\gamma}(\R^d)).
\]
\end{definition}

Thus, solutions to \eqref{eq:01} are understood as mild solutions satisfying the integral equation \eqref{eq:01}.
For the arguments below, we work with the canonical admissible pair
\begin{align*}
  (q,r)=\left( \frac{4(p+1)}{d(p-1)},\, p+1 \right),
\end{align*}
for which the nonlinear estimates required for local well-posedness hold.

\subsection{Strichartz estimates and well-posedness}

We first recall the Strichartz estimates.

\begin{lemma}[Strichartz estimates; see, e.g., \cites{S77,GV85,Y87,KT98}]
Let $(\rho,\gamma)$ and $(\widetilde{\rho},\widetilde{\gamma})$ be admissible
pairs. Then, for any interval $I \ni 0$,
\begin{align*}
\norm{U(\cdot)f}_{L^{\rho}(\R;L^{\gamma}(\R^d))}
&\le C_{0}\norm{f}_{L^{2}(\R^d)}, \\
\norm*{\int_{0}^{\cdot} U(\cdot-s)F(s)\,ds}_{L^{\rho}(I;L^{\gamma}(\R^d))}
&\le C_{0}\norm{F}_{L^{\widetilde{\rho}'}(I;L^{\widetilde{\gamma}'}(\R^d))},
\end{align*}
where $C_{0}>0$ is independent of $I$, and $\eta'$ denotes the dual exponent
of $\eta \ge 1$.
\end{lemma}

For well-posedness results for \eqref{nls}, we refer to \cite{Kato87} for $H^{1}$-solutions, \cites{CW92,GOV94} for $\Sigma$-solutions, and \cite{Cazenave} for a comprehensive account.
For $H^{s}$-solutions with $0<s<1$, local well-posedness was first established by Cazenave and Weissler \cite{CW90} using Besov-type auxiliary spaces, and later refined by Kato \cites{Kato95,Kato96} using Bessel potential spaces, which are also used in the treatment of the momentum $P(u)$ below.



\subsection{Properties of solutions}

We first consider a general property of $H^s$-solutions to \eqref{nls}.
Let $0 \leq s \leq 1$ and let $I=[-T,T]$ be a bounded interval.
Choose an admissible pair $(\rho,\gamma)$ associated with the local theory so that
\[
u \in C(I;H^{s}) \cap L^{\rho}(I;H^{s,\gamma}),
\qquad
f(u)\in L^{\infty}(I;H^{s}) + L^{\rho'}(I;H^{s,\gamma'}).
\]
Since $I$ is bounded, it follows that
\begin{align} \label{eq:f_reg}
f(u)
\in L^{1}(I;H^{s}) + L^{\rho'}(I;H^{s,\gamma'}).
\end{align}
Moreover, for $g \in \mathscr{S}(\R^d)$, where $\mathscr{S}(\R^d)$ denotes the Schwartz space on $\R^d$,
\begin{align*}
\Re(\nabla g,f(g))
= \frac12 \int_{\R^d} \nabla(V(g))\,dx = 0,
\end{align*}
where we used \eqref{A3} and the divergence theorem.
Consequently, if $1/2 \le s \le 1$, then
\begin{align} \label{eq:density_cancel}
\Re \ev*{\nabla u(t), \overline{f(u(t))}} = 0
\end{align}
for almost every $t \in I$.
Here the duality pairing in \eqref{eq:density_cancel} is understood on
$(H^{-1/2}\cap H^{-1/2,\gamma}) \times (H^{1/2}+H^{1/2,\gamma'})$
(or on $L^{2}\times L^{2}$ if $s=1$).
This follows from a standard density argument.

We now return to the fixed pair
\[
  (q,r)=\left(\frac{4(p+1)}{d(p-1)},\,p+1\right),
\]
which is used below in the $H^1$- and $\Sigma$-theory.
As observed in \cite{IMMO25}, such solutions satisfy the following additional regularity properties for every bounded interval $I=[-T,T]$.
First, by \eqref{eq:f_reg} with $s=1$ and
the identity
\[
  J(t) \coloneqq x+it\nabla = e^{\frac{i|x|^{2}}{2t}} (it\nabla) e^{-\frac{i|x|^{2}}{2t}},
\]
we have
\begin{align*}
  f(u) &\in L^1(I;H^{1}) + L^{q'}(I;H^{1,r'}), \\
  J(\cdot)f(u) &\in L^1(I;L^{2}) + L^{q'}(I;L^{r'}).
\end{align*}
Since $xu = J(t)u-it\nabla u \in L^{2}$ for every $t$, any $\Sigma$-solution
also satisfies
\begin{align}
\begin{aligned}
  xu &\in C(I;L^{2}) \cap L^{q}(I;L^{r}), \\
  xf(u) &\in L^1(I;L^{2}) + L^{q'}(I;L^{r'}).
\end{aligned}
\label{reg:4}
\end{align}
Moreover, every $H^{1}$-solution satisfies
\begin{align*}
  f(u),\ \Delta u,\ \partial_t u
  \in C(I;H^{-1}) \cap L^{q}(I;H^{-1,r}).
\end{align*}
Indeed, the restriction on $p$ in \eqref{p:con} implies that
$H^{1} \hookrightarrow L^{r}$ and hence $L^{r'} \hookrightarrow H^{-1}$, so that
$f(u)\in C([-T,T];H^{-1})$.
Combining this with the fact that $U(t)$ is a strongly continuous unitary group
on $H^{-1}$, the integral equation \eqref{eq:01} implies that $u$ is strongly
differentiable in $t$ as an $H^{-1}$-valued function and satisfies
\eqref{nls} in $H^{-1}$.
Furthermore, by \eqref{p:con}, we have
$H^{1,r'} \hookrightarrow L^{2} \hookrightarrow H^{-1,r}$.
Hence, by $H^{1}\hookrightarrow L^{r}$ and H\"older's inequality,
$f(u)\in L^{q}(I;H^{-1,r})$, and therefore
\[
  \partial_t u \in C(I;H^{-1}) \cap L^{q}(I;H^{-1,r}).
\]

\subsection{Auxiliary identities for the scaling (dilation) structure}
\label{sec:auxiliary_identities}

For a function $v$, set
\[
  v_\lambda(x)=\lambda^{d/2}v(\lambda x)
\]
for $\lambda > 0$.
Then
\[
  \left.\frac{d}{d\lambda} v_\lambda \right|_{\lambda=1}
  =
  \left(x\cdot\nabla+\frac d2\right)v.
\]
Accordingly, we define
\[
\mathcal{V}(v)=\int_{\R^d}V(v)\,dx,
\qquad
\mathcal{D}_{\mathrm{sc}}(v)
=
\left.\frac{d}{d\lambda}
\mathcal{V}(v_\lambda) \right|_{\lambda=1}.
\]
Thus, $\mathcal{D}_{\mathrm{sc}}(v)$ is the derivative of the potential energy
along the $L^2$-preserving scaling (dilation) generated by $x\cdot\nabla+d/2$.
Equivalently,
\[
  \mathcal{D}_{\mathrm{sc}}(v)
  = \left.\frac{d}{d\lambda} \int_{\R^d} \lambda^{-d} V(\lambda^{d/2} v(x))\, dx \right|_{\lambda=1}
=
-d\int_{\R^d}V(v)\,dx
+
\frac d2\int_{\R^d}V'(|v|)|v|\,dx.
\]
We also define
\begin{align}
  \mathcal{W}(v)=2\mathcal{V}(v)-\mathcal{D}_{\mathrm{sc}}(v).
  \label{eq:W_def}
\end{align}

The key point is that the same scaling generator appears both the potential
energy and the $J$-interactions. Indeed,
\begin{align*}
(\nabla\cdot x + x\cdot\nabla)v
&=
2\left(x\cdot\nabla+\frac d2\right)v
=
2\left.\frac{d}{d\lambda} v_\lambda \right|_{\lambda=1},
\\
J(s)^2v
&=
|x|^2v+is(\nabla\cdot x+x\cdot\nabla)v-s^2\Delta v.
\end{align*}
Remark that it follows from \eqref{A3} that
\begin{align*}
\Re \ev*{(\nabla\cdot x + x\cdot\nabla)v,\overline{f(v)}}
&= \mathcal{D}_{\mathrm{sc}}(v).
\end{align*}
Using these identities,
together with a standard density argument, we obtain
the following formulas.
If
$v \in H^{1} \cap H^{1,r}$ and $xv \in L^{2} \cap L^{r}$, then,
for each $s \in \R$,
\begin{align}
  \Im \ev*{J(s) v,\overline{J(s) f(v)}}
  &=
  s\,\mathcal{D}_{\mathrm{sc}}(v)
  +
  s^2\Im \ev*{\nabla v,\overline{\nabla f(v)}},
  \label{eq:J_pseudo_algebra}
  \\
  \Re \ev*{\nabla v,\overline{J(s) f(v)}}
  -
  \Re \ev*{J(s) v,\overline{\nabla f(v)}}
  &=
  \mathcal{D}_{\mathrm{sc}}(v)
  +
  2s\Im \ev*{\nabla v,\overline{\nabla f(v)}}.
  \label{eq:J_cross_algebra}
\end{align}
All duality pairings above are understood on
$(L^{2} \cap L^{r}) \times (L^{2} + L^{r'})$.
We use these identities only through the integrated formulas stated below.

Lemma \ref{lem:im_grad} is the only place where the equation \eqref{nls} is
used explicitly.
\begin{lemma}[Lemma 3 in \cite{IMMO25}]
\label{lem:im_grad}
Let $u$ be an $H^{1}$-solution to \eqref{nls}. Then
\begin{align*}
\Im \ev*{ \nabla u, \overline{\nabla f(u)} }_{(L^{2} \cap L^{r}) \times (L^{2} + L^{r'})}
=
2 \Re \ev*{\partial_{t} u, \overline{f(u)}}_{(H^{-1} \cap H^{-1, r}) \times (H^{1} + H^{1, r'})}.
\end{align*}
\end{lemma}

Fix a bounded interval $I=[-T,T]$.
We next state a potential-energy lemma.
The first identity \eqref{eq:potential_time_derivative} was already proven in \cite{IMMO25}.
Since $t\mapsto \mathcal{V}(v(t))$ belongs to $W^{1,1}(I)$, the weighted
identity \eqref{eq:potential_weighted_integration} follows immediately from
integration by parts.

\begin{lemma}[Potential-energy calculus]
Assume that $f$ satisfies \eqref{A1} and \eqref{A3}.
Let $v \in C(I; H^{1}) \cap C^{1}(I;H^{-1})$ with
$\nabla v \in L^{q}(I; L^{r})$ and
$\partial_{t}v \in L^{q}(I; H^{-1,r})$.
Then $t\mapsto \mathcal{V}(v(t))$
belongs to $W^{1,1}(I)$, and
\begin{align}
\frac{d}{dt}\mathcal{V}(v(t))
=
2\Re \ev*{\partial_t v(t),\overline{f(v(t))}}_
{(H^{-1}\cap H^{-1,r})\times (H^{1}+H^{1,r'})}
\label{eq:potential_time_derivative}
\end{align}
holds in $L^{1}(I)$.
Moreover, for any $t\in I$ and any integer $k\ge 1$,
\begin{align}
  2\int_0^t s^k
  \Re \ev*{\partial_t v(s),\overline{f(v(s))}}\,ds
  = t^k\mathcal{V}(v(t))
  - k \int_0^t s^{k-1}\mathcal{V}(v(s))\,ds.
\label{eq:potential_weighted_integration}
\end{align}
\end{lemma}

The following integrated identities are obtained from the above auxiliary identities involving $\mathcal{D}_{\mathrm{sc}}(v)$ together with the potential-energy calculus.
\begin{lemma}[Integrated identities involving the $J$-terms]
\label{lem:integrated_J_terms}
Let $u$ be a $\Sigma$-solution to \eqref{nls}. Then, for any $t\in I$,
\begin{align}
  \int_0^t
  \Im \ev*{J(s)u(s),\overline{J(s)f(u(s))}}\,ds
  =
  t^2\mathcal{V}(u(t))
  -
  \int_0^t s\,\mathcal{W}(u(s))\,ds.
  \label{eq:integrated_pseudo_term}
\end{align}
Moreover,
\begin{align}
\begin{aligned}
  &\int_{0}^{t}
  \left(
  \Re \ev*{\nabla u(s),\overline{J(s)f(u(s))}}
  -
  \Re \ev*{J(s)u(s),\overline{\nabla f(u(s))}}
  \right)ds \\
  ={}&
  -\int_0^t \mathcal{W}(u(s))\,ds
  +
  2t\mathcal{V}(u(t)).
\end{aligned}
  \label{eq:integrated_cross_term}
\end{align}
\end{lemma}

\begin{proof}[Proof of Lemma \ref{lem:integrated_J_terms}]
Integrating \eqref{eq:J_pseudo_algebra} in time, applying
Lemma \ref{lem:im_grad}, and then using
\eqref{eq:potential_weighted_integration} for $k=2$ together with \eqref{eq:W_def}, we obtain
\eqref{eq:integrated_pseudo_term}.
Similarly, integrating \eqref{eq:J_cross_algebra} in time, applying
Lemma \ref{lem:im_grad}, and then using
\eqref{eq:potential_weighted_integration} for $k=1$ together with \eqref{eq:W_def}, we obtain
\eqref{eq:integrated_cross_term}.
\end{proof}

We conclude this section with a technical exchange lemma for Duhamel terms, which will be needed in the proof of Proposition \ref{prop:main}.

\begin{lemma}[cf. Lemma 6 in \cite{IMMO25}]
\label{lem:integral_exchange}
Let $\gamma \ge 1$ and $\rho \ge 1$. Assume that there exists an integer $\sigma \ge 0$ such that
$H^{\sigma}(\R^d) \hookrightarrow L^{\gamma}(\R^d)$ continuously.
Let $\psi \in L^{2} \cap L^\gamma$ and
$g_{j} \in L^{1}(I; L^{2}) + L^{\rho'}(I; L^{\gamma'})$ for $j=1,2$.
Then, for any $t \in I$,
\begin{align*}
\left( \psi, \int_0^t U(-s) g_{1}(s)\, ds \right)
=
\int_{0}^{t}
\ev{ U(s) \psi, \overline{g_{1}(s)} }\, ds,
\end{align*}
where the duality pairing is understood on
$(L^{2} \cap L^\gamma) \times (L^{2} + L^{\gamma'})$, and
\begin{align*}
&\left( \int_0^t U(-s) g_{1}(s)\, ds, \int_0^t U(-\theta) g_{2}(\theta)\, d\theta \right) \\
={}&
\int_{0}^{t}
\ev*{ \int_{0}^{\theta} U(\theta-s) g_{1}(s)\, ds, \overline{g_{2}(\theta)} }\, d \theta
+
\int_{0}^{t}
\ev*{g_{1}(s), \overline{\int_{0}^{s} U(s-\theta) g_{2}(\theta)\, d\theta} }\, ds,
\end{align*}
where the duality pairings are again on
$(L^{2} \cap L^\gamma) \times (L^{2} + L^{\gamma'})$.
\end{lemma}

\begin{proof}[Proof of Lemma \ref{lem:integral_exchange}]
We prove only the second assertion, since the first is similar and simpler.
Using the Yosida approximation $(1-\varepsilon\Delta)^{-\sigma}$ and the dual embedding
$L^{\gamma'} \hookrightarrow H^{-\sigma}$,
we obtain
\[
(1-\varepsilon\Delta)^{-\sigma} g_{1},\;
(1-\varepsilon\Delta)^{-\sigma} g_{2} \in L^{1}(I;L^{2}),
\]
This implies that $U(-s)(1- \varepsilon \Delta)^{-\sigma} g_{1}(s) \overline{U(-\theta) (1- \varepsilon \Delta)^{-\sigma} g_{2}(\theta)}$
is integrable on $[0,t]^{2}\times\R^{d}$.
The conclusion follows by the unitarity of $U(t)$ on $L^{2}$, Fubini's theorem, and the limit $\varepsilon \downarrow 0$.
\end{proof}

\section{Systematic derivation of conservation laws}
\label{sec:unified_identity}

In this section, we show that the conservation laws discussed above can be derived systematically from a single integral identity.
By the regularity properties established in Section \ref{sec:prelim}, all duality pairings $\ev*{\cdot,\cdot}$ at fixed times, as well as the corresponding time integrals appearing below, are well defined.
Unless otherwise stated, these pairings are understood on
\[
(L^{2}\cap L^{r}) \times (L^{2}+L^{r'})
\]
and, for the time-integrated terms, on
\[
\left(L^{\infty}_{t}L^{2}\cap L^{q}_{t}L^{r}\right)
\times
\left(L^{1}_{t}L^{2}+L^{q'}_{t}L^{r'}\right),
\]
where $(q,r)=\left(\frac{4(p+1)}{d(p-1)},\,p+1\right)$.
Within this framework, we may work directly with the integral equations without repeatedly restating the underlying function spaces.

The following proposition is the key ingredient in our unified approach.
It provides a general identity from which all of the conservation laws and identities considered in this paper can be derived in a systematic way.

\begin{proposition} \label{prop:main}
Let $I$ be an interval containing $0$.
Let $\gamma \ge 1$ and $\rho \ge 1$, and assume that
$H^{\sigma}(\R^d)\hookrightarrow L^{\gamma}(\R^d)$ for some
integer $\sigma \ge 0$.
Suppose that $\psi_{j}\in L^{2}\cap L^{\gamma}$ and
$g_{j}\in L^{1}(I;L^{2}) + L^{\rho'}(I;L^{\gamma'})$ for $j=1,2$.
Define $v_{j}(t)$ by
\begin{align*}
  v_{j}(t)
  = U(t)\psi_{j} - i\int_{0}^{t} U(t-s)g_{j}(s)\,ds,
  \qquad t \in I.
\end{align*}
Then, for every $t \in I$,
\begin{align} \label{eq:main}
  (v_{1}(t),v_{2}(t))
  &= (\psi_{1},\psi_{2})
  + i\int_{0}^{t}
  \left(
  \ev*{v_{1}(s),\overline{g_{2}(s)}} - \ev*{\overline{v_{2}(s)},g_{1}(s)}
  \right)\,ds.
\end{align}

In particular, if $\psi_{1}=\psi_{2}=\psi$ and $g_{1}=g_{2}=g$, then
$v_{1}=v_{2}=v$ satisfies
\begin{align} \label{eq:case1}
  \norm{v(t)}_{L^{2}}^{2}
  = \norm{\psi}_{L^{2}}^{2}
  + 2\Im \int_{0}^{t} \ev*{\overline{v(s)},g(s)}\,ds,
\end{align}
where the time-integrated duality pairing is understood on
\[
  \left(L^{\infty}_{t}L^{2}\cap L^{\rho}_{t}L^{\gamma}\right)
  \times
  \left(L^{1}_{t}L^{2}+L^{\rho'}_{t}L^{\gamma'}\right).
\]
\end{proposition}

\begin{proof}
We calculate the inner product $(v_{1}(t), v_{2}(t))$.
By the unitarity of $U(t)$, we have
\begin{align*}
(v_{1}(t), v_{2}(t))
&= (U(-t)v_{1}(t), U(-t)v_{2}(t)) \\
&= (\psi_{1}, \psi_{2}) \\
&\quad + \left( \psi_{1}, -i \int_{0}^{t} U(-\tau) g_{2}(\tau) \, d\tau \right)
+ \left( -i \int_{0}^{t} U(-s) g_{1}(s) \, ds, \psi_{2} \right) \\
&\quad + \left( -i \int_{0}^{t} U(-s) g_{1}(s) \, ds, -i \int_{0}^{t} U(-\tau) g_{2}(\tau) \, d\tau \right).
\end{align*}
Applying Lemma \ref{lem:integral_exchange} to exchange the integrals and the inner products, we obtain
\begin{align*}
(v_{1}(t), v_{2}(t))
  &= (\psi_{1}, \psi_{2})
+ i \int_{0}^{t} \ev*{ U(\tau) \psi_{1}, \overline{g_{2}(\tau)} } \, d\tau
- i \int_{0}^{t} \ev*{ \overline{U(s) \psi_{2}}, g_{1}(s) } \, ds \\
  &\quad + i \int_{0}^{t} \ev*{ -i \int_{0}^{\tau} U(\tau-s) g_{1}(s) \, ds, \overline{g_{2}(\tau)} } \, d\tau \\
  &\quad - i \int_{0}^{t} \ev*{ \overline{-i \int_{0}^{s} U(s-\tau) g_{2}(\tau) \, d\tau}, g_{1}(s) } \, ds \\
  &= (\psi_{1}, \psi_{2})
+ i \int_{0}^{t} \ev*{ U(\tau) \psi_{1} -i \int_{0}^{\tau} U(\tau-s) g_{1}(s) \, ds, \overline{g_{2}(\tau)} } \, d\tau \\
  &\quad - i \int_{0}^{t} \ev*{ \overline{U(s) \psi_{2} -i \int_{0}^{s} U(s-\tau) g_{2}(\tau) \, d\tau}, g_{1}(s) } \, ds \\
  &= (\psi_{1}, \psi_{2})
+ i \int_{0}^{t} \ev*{ v_{1}(\tau), \overline{g_{2}(\tau)} } \, d\tau
- i \int_{0}^{t} \ev*{ \overline{v_{2}(s)}, g_{1}(s) } \, ds.
\end{align*}
Thus \eqref{eq:main} holds. The particular case \eqref{eq:case1} follows immediately by substituting $v_{1} = v_{2} = v$, $\psi_{1} = \psi_{2} = \psi$, and $g_{1} = g_{2} = g$ into \eqref{eq:main}.
\end{proof}

\begin{remark}
Although the identity \eqref{eq:main} is formulated here in terms of Lebesgue
spaces, its underlying mechanism is essentially abstract.
The crucial ingredient is not the specific Lebesgue-space
setting itself, but rather the existence of a sufficiently robust dual-pairing
framework compatible with the dispersive or smoothing estimates of the underlying
evolution.
This suggests that, after replacing $L^\gamma$ by suitable Banach spaces
in an appropriate Gelfand triple, analogous identities should persist for a broader
class of dispersive or wave equations.
\end{remark}

\subsection{Charge and energy conservation}

The third author \cite{Ozawa06} first showed that charge and energy conservation can be derived directly from the integral equation.
Within the framework of Proposition \ref{prop:main}, these laws are recovered in a unified way.
\begin{proposition}[Charge conservation]
Let $f$ satisfy \eqref{A1} and \eqref{A2}.
Take $\phi \in L^2(\R^d)$. Let $u$ be an $L^{2}$-solution of \eqref{eq:01} on $[-T,T]$ for some $T>0$.
Then $\norm{u(t)}_{L^2} = \norm{\phi}_{L^2}$ for all $t \in [-T, T]$.
\end{proposition}
\begin{proof}
The integral equation \eqref{eq:01} with $t_{0}=0$ reads
\begin{align*}
u(t)=U(t)\phi - i\int_{0}^{t} U(t-s)f(u(s))\,ds.
\end{align*}
Applying \eqref{eq:case1} with $v(t)=u(t)$, $\psi=\phi$, and
$g(s)=f(u(s))$, we obtain
\begin{align*}
  \norm{u(t)}_{L^{2}}^{2}
  = \norm{\phi}_{L^{2}}^{2}
  + 2\Im \int_{0}^{t} \ev*{\overline{u(s)},f(u(s))}\,ds.
\end{align*}
Since $\Im(\overline{z}f(z))=0$ by \eqref{A2}, the last term vanishes.
Hence $\norm{u(t)}_{L^{2}}=\norm{\phi}_{L^{2}}$ for all $t\in[-T,T]$.
\end{proof}

For $H^{1}$-solutions, we define the energy by
\[
  E(u)=\frac12 \norm{\nabla u}_{L^{2}}^{2} + \int_{\R^d} V(u)\,dx.
\]

\begin{proposition}[Energy conservation]
Let $f$ satisfy \eqref{A1} and \eqref{A3}.
Take $\phi \in H^1(\R^d)$. Let $u$ be an $H^{1}$-solution of \eqref{eq:01} on $[-T,T]$ for some $T>0$.
Then $E(u(t)) = E(\phi)$ for all $t \in [-T, T]$.
\end{proposition}
\begin{proof}
Applying $\nabla$ to \eqref{eq:01} with $t_{0} = 0$, we obtain
\begin{align*}
  \nabla u(t)
  = U(t)\nabla \phi - i\int_{0}^{t} U(t-s)\nabla f(u(s))\,ds.
\end{align*}
Applying \eqref{eq:case1} with $v(t)=\nabla u(t)$, $\psi=\nabla\phi$, and
$g(s)=\nabla f(u(s))$, we get
\begin{align*}
  \norm{\nabla u(t)}_{L^{2}}^{2}
  = \norm{\nabla \phi}_{L^{2}}^{2}
  + 2\Im \int_{0}^{t} \ev*{\overline{\nabla u(s)},\nabla f(u(s))}\,ds.
\end{align*}
By Lemma \ref{lem:im_grad} and \eqref{eq:potential_time_derivative},
\begin{align*}
  2\Im \int_{0}^{t} \ev*{\overline{\nabla u(s)},\nabla f(u(s))}\,ds
  &= -2\int_{0}^{t} \frac{d}{ds}
  \left(\int_{\R^{d}} V(u(s))\,dx\right)\,ds \\
  &= -2\left(
  \int_{\R^{d}} V(u(t))\,dx - \int_{\R^{d}} V(\phi)\,dx
  \right).
\end{align*}
Therefore $E(u(t))=E(\phi)$ for all $t\in[-T,T]$.
\end{proof}

\subsection{Momentum conservation}

For $H^{1}$-solutions,
the momentum is defined by
\[
P(u)=\Im\int_{\R^d}\overline{u}\,\nabla u\,dx.
\]
The direct integral-equation approach to momentum conservation was developed in \cite{FM17}.
Within the framework of Proposition \ref{prop:main}, the momentum conservation
is also derived in a unified manner.

\begin{proposition}[Momentum conservation]
Assume that $f$ satisfies \eqref{A1} -- \eqref{A3}.
Take $\phi \in H^1(\R^d)$.
Let $u$ be an $H^{1}$-solution to \eqref{eq:01} on $[-T,T]$ for some $T>0$.
Then $P(u(t)) = P(\phi)$ for all $t \in [-T, T]$.
\end{proposition}

\begin{proof}
Arguing as in the proof of charge and energy conservation, we obtain
\begin{align*}
u(t)
&= U(t)\phi
- i\int_{0}^{t} U(t-s)f(u(s))\,ds, \\
\nabla u(t)
&= U(t)\nabla \phi
- i\int_{0}^{t} U(t-s)\nabla f(u(s))\,ds.
\end{align*}
Applying \eqref{eq:main} with $v_{1}(t)=u(t)$, $v_{2}(t)=\nabla u(t)$,
$\psi_1=\phi$, $\psi_2=\nabla\phi$,
$g_1(t)=f(u(t))$, and $g_2(t)=\nabla f(u(t))$,
together with integration by parts, we obtain
\begin{align*}
  (u(t),\nabla u(t))
  &= (\phi,\nabla\phi)
  + i\int_0^t
  \left(
  \ev*{u(s),\overline{\nabla f(u(s))}}
  -
  \ev*{\overline{\nabla u(s)},f(u(s))}
  \right)\,ds \\
  &= (\phi,\nabla\phi)
  - 2i\int_0^t \Re \ev*{\nabla u(s),\overline{f(u(s))}}\,ds.
\end{align*}
Hence
\begin{align*}
iP(u(t))
&= iP(\phi)
-2i\int_0^t \Re \ev*{\nabla u(s),\overline{f(u(s))}}\,ds.
\end{align*}
By the cancellation property \eqref{eq:density_cancel},
the integrand on the right-hand side vanishes.
Therefore $P(u(t))=P(\phi)$ for all $t\in[-T,T]$.
\end{proof}

\begin{remark}
Arguing as in \cite{FM17},
the same argument extends to fractional regularity.
More precisely, if $d \ge 2$, let $1/2 \le s < 1$ and assume in addition that
\[
1 < p \le 1+\frac{4}{d-2s},
\]
so that $H^s(\R^d)\hookrightarrow L^{p+1}(\R^d)$.
If $d=1$, let $1/2 < s < 1$. Then $H^s(\R)\hookrightarrow L^\infty(\R)$.
In either case, \eqref{eq:f_reg} follows from the corresponding Sobolev embedding
and the fractional chain rule.
Then
the same proof yields momentum conservation for $H^s$-solutions.
To avoid working directly with the duality pairing on
$H^{-1/2}\times H^{1/2}$, one may insert the Riesz transform
\[
\mathcal{R}\coloneqq -\frac{\nabla}{|\nabla|}.
\]
Using the factorization $\nabla = -\mathcal{R}|\nabla|^{\frac12}|\nabla|^{\frac12}$,
the momentum can be rewritten in the $L^2$-based form
\[
iP(u)=\bigl(|\nabla|^{\frac12}u,\mathcal{R}|\nabla|^{\frac12}u\bigr).
\]
Applying \eqref{eq:main},
we obtain the same conclusion, since the nonlinear
contribution vanishes by the skew-self-adjointness of $\mathcal{R}$.
\end{remark}

\subsection{Pseudo-conformal conservation law and virial-type identities}

We next consider the pseudo-conformal conservation law and the virial-type identities,
which are both governed by the $L^2$-preserving scaling structure.
Within the framework of Proposition \ref{prop:main},
both identities are derived in a unified manner.
For the pseudo-conformal conservation law, we follow the direct
integral-equation approach of \cite{FM17}.

\begin{proposition}[Pseudo-conformal conservation law]
Assume that $f$ satisfies \eqref{A1} -- \eqref{A3}.
Let $u$ be a $\Sigma$-solution to \eqref{nls} on $[-T, T]$ for some $T>0$ with $\phi \in \Sigma$.
Then
\begin{align}
&\norm{J(t)u(t)}_{L^2}^2 + 2t^2 \int_{\R^d}V(u(t))\, dx
= \norm{x \phi}_{L^2}^2 + 2 \int_{0}^{t} \left( s \int_{\R^{d}} W(u(s))\, dx \right) ds \label{pc-law}
\end{align}
for all $t \in [-T, T]$, where $W(u) = (d+2)V(u) - \frac{d}{2}V'(|u|)\abs{u}$.
\end{proposition}

\begin{proof}
Recall $J(t) =U(t)xU(-t)$.
Applying $J(t)$ to the integral equation, we obtain
\begin{align*}
J(t)u(t)
= U(t)x\phi - i\int_{0}^{t} U(t-s)J(s)f(u(s))\,ds.
\end{align*}
Applying \eqref{eq:case1} with $v(t)=J(t)u(t)$, $\psi=x\phi$, and
$g(s)=J(s)f(u(s))$, we obtain
\begin{align*}
\norm{J(t)u(t)}_{L^{2}}^{2}
&= \norm{x\phi}_{L^{2}}^{2}
+ 2\Im \int_{0}^{t}
\ev*{\overline{J(s)u(s)},J(s)f(u(s))}\,ds \\
&= \norm{x\phi}_{L^{2}}^{2}
- 2\int_{0}^{t}
\Im
\ev*{J(s)u(s),\overline{J(s)f(u(s))}}\,ds.
\end{align*}
Substituting \eqref{eq:integrated_pseudo_term} into the identity, we obtain \eqref{pc-law}.
\end{proof}

We next derive the virial-type identities.
Except for the first-order virial identity,
a rigorous derivation in the integral-equation framework was obtained in
\cite{IMMO25}.
In the standard approach, a weight such as
$e^{-\varepsilon|x|^2}$ is typically introduced to recover integrability.
The argument in \cite{IMMO25} avoids this regularization, but still relies on
the pseudo-conformal conservation law.
By contrast, the present approach yields a direct derivation of the first-order
virial identity from \eqref{eq:main}, avoiding both the weight regularization
and the pseudo-conformal conservation law.

\begin{proposition}[First-order
virial identity]
Assume that $f$ satisfies \eqref{A1} -- \eqref{A3}.
Let $u$ be a $\Sigma$-solution to \eqref{nls} on $[-T, T]$ for some $T>0$ with $\phi \in \Sigma$.
Then for any $t \in [-T, T]$,
\begin{align} \label{eq:vi_1st}
\norm{xu(t)}_{L^2}^2 = \norm{x\phi}_{L^2}^2 + 2 \int_0^t \Im ( \nabla u(s), x u(s) ) \, ds.
\end{align}
\end{proposition}

\begin{proof}
We begin with the equation satisfied by $xu$.
Since
\[
i\partial_t(xu)+\frac12 \Delta(xu)=xf(u)+\nabla u
\]
holds in $H^{-2}(\R^d)$ and the right-hand side belongs to
$L^{1}([-T,T];H^{-2}(\R^d))$, the function $xu$ is absolutely continuous in
time as an $H^{-2}(\R^d)$-valued function.
Hence
\begin{align*}
xu(t)
= U(t)x\phi - i\int_{0}^{t} U(t-s)\bigl(\nabla u(s)+x f(u(s))\bigr)\,ds.
\end{align*}
By \eqref{reg:4}, all duality pairings below are well defined.
Applying \eqref{eq:case1} with $v(t)=xu(t)$, $\psi=x\phi$, and
$g(s)=\nabla u(s)+x f(u(s))$, we obtain
\begin{align*}
\norm{xu(t)}_{L^{2}}^{2}
&= \norm{x\phi}_{L^{2}}^{2}
+ 2\Im \int_{0}^{t}
\ev*{x \overline{u(s)},\nabla u(s)+x f(u(s))}\,ds \\
&= \norm{x\phi}_{L^{2}}^{2}
+ 2\int_{0}^{t} \Im(\nabla u(s),xu(s))\,ds,
\end{align*}
where we used
\[
\Im \ev{x\overline{u},xf(u)}
=
\int_{\R^{d}} |x|^{2}\Im(\overline{u}f(u))\,dx = 0
\]
by \eqref{A2}.
This proves \eqref{eq:vi_1st}.
\end{proof}

We next analyze the quantity $\Im(\nabla u,xu)$ appearing in the virial
identity for $\Sigma$-solutions.
The following cross-term identity is obtained from Proposition \ref{prop:main} in the same manner as above.

\begin{proposition}[Cross-term identity]
Assume that $f$ satisfies \eqref{A1} -- \eqref{A3}.
Let $u$ be a $\Sigma$-solution to \eqref{nls} on $[-T, T]$ for some $T>0$ with $\phi \in \Sigma$.
Then
\begin{align} \label{eq:cross_term_identity}
\Im (\nabla u(t), x u(t)) = \Im (\nabla \phi, x\phi) + 2t E(\phi) - \int_0^t \left( \int_{\R^d} W(u(s)) \, dx \right) ds
\end{align}
for any $t \in [-T, T]$, where $W(u) = (d+2)V(u) - \frac{d}{2}V'(|u|)\abs{u}$.
\end{proposition}

\begin{proof}
We consider the two integral equations
\begin{align*}
J(t)u(t)
&= U(t)x\phi - i\int_{0}^{t} U(t-s)J(s)f(u(s))\,ds
\end{align*}
and
\begin{align*}
\nabla u(t)
&= U(t)\nabla\phi - i\int_{0}^{t} U(t-s)\nabla f(u(s))\,ds.
\end{align*}
Applying \eqref{eq:main} with
\[
  v_{1}(t)=\nabla u(t),\quad v_{2}(t)=J(t)u(t),\quad
  \psi_{1}=\nabla\phi,\quad \psi_{2}=x\phi,
\]
\[
g_{1}(s)=\nabla f(u(s)),\quad g_{2}(s)=J(s)f(u(s)),
\]
and taking the imaginary part,
we obtain
\begin{align*}
\Im(\nabla u(t),J(t)u(t))
={}& \Im(\nabla\phi,x\phi) \\
&+
\int_{0}^{t}
\left(
\Re \ev*{\nabla u(s),\overline{J(s)f(u(s))}}
-
\Re \ev*{J(s)u(s),\overline{\nabla f(u(s))}}
\right)ds.
\end{align*}
Since $J(t) =x +it\nabla$, we have
\[
\Im(\nabla u(t),J(t)u(t))
=
\Im(\nabla u(t),xu(t))
-
t\norm{\nabla u(t)}_{L^2}^{2}.
\]
Therefore, combining the two identities above with \eqref{eq:integrated_cross_term},
\begin{align*}
  \Im(\nabla u(t),xu(t))
  &=
  \Im(\nabla u(t),J(t)u(t))
  +
  t\norm{\nabla u(t)}_{L^2}^{2} \\
  &=
  \Im(\nabla\phi,x\phi)
  -
  \int_0^t \left( \int_{\R^{d}}W(u(s))\, dx \right) ds
  +
  t\norm{\nabla u(t)}_{L^2}^{2}
  +
  2t \int_{\R^{d}}V(u(t))\, dx.
\end{align*}
Using the conservation of energy, we get
\[
\norm{\nabla u(t)}_{L^2}^{2}
+
2 \int_{\R^{d}}V(u(t))\, dx
=
2E(u(t))
=
2E(\phi).
\]
Substituting this into the preceding identity yields \eqref{eq:cross_term_identity}.
\end{proof}

Combining the preceding two propositions, we obtain the virial identity.

\begin{proposition}[Virial identity] 
Assume that $f$ satisfies \eqref{A1} -- \eqref{A3}.
Let $u$ be a $\Sigma$-solution to \eqref{nls} on $[-T, T]$ for some $T>0$ with $\phi \in \Sigma$.
Then
\begin{align*}
  \norm{xu(t)}_{L^2}^2 ={}& \norm{x\phi}_{L^2}^2 + 2t \Im (\nabla \phi, x\phi) + 2t^2 E(\phi)
  - 2 \int_0^t \int_0^s \int_{\R^d} W(u(\tau)) \, dx d\tau ds
\end{align*}
for any $t \in [-T, T]$.
\end{proposition}

\begin{proof}
Substituting \eqref{eq:cross_term_identity} into \eqref{eq:vi_1st}, we obtain
\begin{align*}
\norm{xu(t)}_{L^2}^2
&= \norm{x\phi}_{L^2}^2
+ 2 \int_0^t
\left(
\Im (\nabla \phi, x\phi) + 2s E(\phi)
- \int_0^s \int_{\R^d} W(u(\tau)) \, dx\, d\tau
\right) ds \\
&= \norm{x\phi}_{L^2}^2
+ 2t \Im (\nabla \phi, x\phi) + 2t^2 E(\phi)
- 2 \int_0^t \int_0^s \int_{\R^d} W(u(\tau)) \, dx\, d\tau\, ds.
\end{align*}
\end{proof}

\section{Conclusion}

We have developed a unified integral-equation approach to the rigorous
justification of conservation laws and related identities for nonlinear
Schr\"odinger equations.
Its central ingredient is the master identity in Proposition \ref{prop:main},
from which we derive, within a unified framework, the conservation of
charge, energy, and momentum, together with the pseudo-conformal conservation
law and the virial-type identities.

A notable feature of this approach is that it works directly at the level of
mild solutions, without relying on approximation by smooth solutions or on
regularization of the equation.
In this sense, the master identity clarifies the duality
structure underlying these formal identities.

A natural direction for future work is to investigate whether this method extends beyond the
present setting to broader classes of dispersive or wave equations.
Related developments for singular interactions and quantum hydrodynamic
systems suggest that integral-equation arguments are particularly relevant when
approximation-based derivations become delicate.
It would therefore be interesting to understand whether, in more general
settings, analogous identities can still be recovered once a suitable
functional framework for mild solutions and dual pairings is available.

\section*{Acknowledgments}
H.M. was supported by JSPS KAKENHI Grant Number 22K13941 and 26K00612.
T.O. was supported by JSPS KAKENHI Grant Number 24H00024 and by JST Moonshot R\&D Grant Number JPMJMS25A5.



\begin{bibdiv}
\begin{biblist}


\bib{AT01}{article}{
   author={Adami, Riccardo},
   author={Teta, Alessandro},
   title={A class of nonlinear Schr\"odinger equations with concentrated
   nonlinearity},
   journal={J. Funct. Anal.},
   volume={180},
   date={2001},
   number={1},
   pages={148--175},
   issn={0022-1236},
   review={\MR{1814425}},
}

\bib{AMS24}{article}{
   author={Antonelli, Paolo},
   author={Marcati, Pierangelo},
   author={Scandone, Raffaele},
   title={Existence and stability of almost finite energy weak solutions to
   the quantum Euler-Maxwell system},
   language={English, with English and French summaries},
   journal={J. Math. Pures Appl. (9)},
   volume={191},
   date={2024},
   pages={Paper No. 103629, 38},
   issn={0021-7824},
   review={\MR{4823114}},
}

\bib{A21}{article}{
   author={Antonelli, Paolo},
   title={Remarks on the derivation of finite energy weak solutions to the
   QHD system},
   journal={Proc. Amer. Math. Soc.},
   volume={149},
   date={2021},
   number={5},
   pages={1985--1997},
   issn={0002-9939},
   review={\MR{4232191}},
}

\bib{Cazenave}{book}{
author={Cazenave, Thierry},
title={Semilinear {S}chr\"{o}dinger {E}quations},
series={Courant Lecture Notes in Mathematics},
publisher={New York University, Courant Institute of Mathematical Sciences, New York; American Mathematical Society, Providence, RI},
date={2003},
volume={10},
ISBN={0-8218-3399-5},
url={https://doi.org/10.1090/cln/010},
review={\MR{2002047}},
}

\bib{CW90}{article}{
author={Cazenave, Thierry},
author={Weissler, Fred B.},
title={The Cauchy problem for the critical nonlinear Schr\"odinger
equation in $H^s$},
journal={Nonlinear Anal.},
volume={14},
date={1990},
number={10},
pages={807--836},
issn={0362-546X},
review={\MR{1055532}},
}

\bib{CW92}{article}{
author={Cazenave, Thierry},
author={Weissler, Fred~B.},
title={Rapidly decaying solutions of the nonlinear {S}chr\"{o}dinger equation},
date={1992},
ISSN={0010-3616},
journal={Comm. Math. Phys.},
volume={147},
number={1},
pages={75\ndash 100},
url={http://projecteuclid.org/euclid.cmp/1104250527},
review={\MR{1171761}},
}

\bib{FM17}{article}{
author={Fujiwara, Kazumasa},
author={Miyazaki, Hayato},
title={The derivation of conservation laws for nonlinear {S}chr\"{o}dinger equations with power type nonlinearities},
date={2017},
journal={RIMS K\^{o}ky\^{u}roku Bessatsu, B63},
publisher={Res. Inst. Math. Sci. (RIMS), Kyoto},
pages={13\ndash 21},
review={\MR{3751978}},
}



\bib{Ginibre97}{article}{
   author={Ginibre, J.},
   title={An introduction to nonlinear Schr\"odinger equations},
   conference={
      title={Nonlinear waves},
      address={Sapporo},
      date={1995},
   },
   book={
      series={GAKUTO Internat. Ser. Math. Sci. Appl.},
      volume={10},
      publisher={Gakk\={o}tosho, Tokyo},
   },
   isbn={4-7625-0419-X},
   date={1997},
   pages={85--133},
   review={\MR{1602772}},
}

\bib{GOV94}{article}{
author={Ginibre, J.},
author={Ozawa, T.},
author={Velo, G.},
title={On the existence of the wave operators for a class of nonlinear {S}chr\"{o}dinger equations},
date={1994},
ISSN={0246-0211},
journal={Ann. Inst. H. Poincar\'{e} Phys. Th\'{e}or.},
volume={60},
number={2},
pages={211\ndash 239},
review={\MR{1270296}},
}

\bib{GV79}{article}{
author={Ginibre, J.},
author={Velo, G.},
title={On a class of nonlinear Schr\"odinger equations. I. The Cauchy problem, general case},
journal={J. Functional Analysis},
volume={32},
date={1979},
number={1},
pages={1--32},
issn={0022-1236},
review={\MR{0533218}},
}

\bib{GV85}{article}{
author={Ginibre, J.},
author={Velo, G.},
title={The global {C}auchy problem for the nonlinear {S}chr\"{o}dinger equation revisited},
date={1985},
ISSN={0294-1449},
journal={Ann. Inst. H. Poincar\'{e} Anal. Non Lin\'{e}aire},
volume={2},
number={4},
pages={309\ndash 327},
url={http://www.numdam.org/item?id=AIHPC_1985__2_4_309_0},
review={\MR{801582}},
}

\bib{Glassey77}{article}{
author={Glassey, R.~T.},
title={On the blowing up of solutions to the {C}auchy problem for nonlinear {S}chr\"{o}dinger equations},
date={1977},
ISSN={0022-2488},
journal={J. Math. Phys.},
volume={18},
number={9},
pages={1794\ndash 1797},
url={https://doi.org/10.1063/1.523491},
review={\MR{460850}},
}


\bib{IMMO25}{article}{
author={Ikeda, Tomoyuki},
author={Machihara, Shuji},
author={Miyazaki, Hayato},
author={Ozawa, Tohru},
title={Remarks on the derivation of the virial identity for nonlinear {S}chr\"{o}dinger equations},
journal={preprint},
eprint={arXiv:2510.18168},
url={https://arxiv.org/abs/2510.18168},
}

\bib{Kato87}{article}{
author={Kato, Tosio},
title={On nonlinear {S}chr\"{o}dinger equations},
date={1987},
ISSN={0246-0211},
journal={Ann. Inst. H. Poincar\'{e} Phys. Th\'{e}or.},
volume={46},
number={1},
pages={113\ndash 129},
url={http://www.numdam.org/item?id=AIHPB_1987__46_1_113_0},
review={\MR{877998}},
}

\bib{Kato89}{article}{
   author={Kato, Tosio},
   title={Nonlinear Schr\"odinger equations},
   conference={
      title={Schr\"odinger operators},
      address={S\o nderborg},
      date={1988},
   },
   book={
      series={Lecture Notes in Phys.},
      volume={345},
      publisher={Springer, Berlin},
   },
   isbn={3-540-51783-9},
   date={1989},
   pages={218--263},
   review={\MR{1037322}},
}

\bib{Kato95}{article}{
author={Kato, Tosio},
title={On nonlinear Schr\"odinger equations. II. $H^s$-solutions and
unconditional well-posedness},
journal={J. Anal. Math.},
volume={67},
date={1995},
pages={281--306},
issn={0021-7670},
review={\MR{1383498}},
}

\bib{Kato96}{article}{
author={Kato, Tosio},
title={Correction to: ``On nonlinear Schr\"odinger equations. II.
$H^s$-solutions and unconditional well-posedness''},
journal={J. Anal. Math.},
volume={68},
date={1996},
pages={305},
issn={0021-7670},
review={\MR{1403260}},
}

\bib{KT98}{article}{
author={Keel, Markus},
author={Tao, Terence},
title={Endpoint {S}trichartz estimates},
date={1998},
ISSN={0002-9327},
journal={Amer. J. Math.},
volume={120},
number={5},
pages={955\ndash 980},
url={http://muse.jhu.edu/journals/american_journal_of_mathematics/v120/120.5keel.pdf},
review={\MR{1646048}},
}

\bib{MR05}{article}{
author={Merle, Frank},
author={Rapha\"{e}l, Pierre},
title={The blow-up dynamic and upper bound on the blow-up rate for critical nonlinear {S}chr\"{o}dinger equation},
date={2005},
ISSN={0003-486X},
journal={Ann. of Math. (2)},
volume={161},
number={1},
pages={157\ndash 222},
url={https://doi.org/10.4007/annals.2005.161.157},
review={\MR{2150386}},
}

\bib{Ozawa06}{article}{
author={Ozawa, T.},
title={Remarks on proofs of conservation laws for nonlinear {S}chr\"{o}dinger equations},
date={2006},
ISSN={0944-2669},
journal={Calc. Var. Partial Differential Equations},
volume={25},
number={3},
pages={403\ndash 408},
url={https://doi.org/10.1007/s00526-005-0349-2},
review={\MR{2201679}},
}



\bib{S25}{article}{
   author={Scandone, Raffaele},
   title={Global, finite energy, weak solutions to a magnetic quantum fluid
   system},
   journal={Netw. Heterog. Media},
   volume={20},
   date={2025},
   number={2},
   pages={345--355},
   issn={1556-1801},
   review={\MR{4897328}},
}

\bib{S77}{article}{
author={Strichartz, Robert~S.},
title={Restrictions of {F}ourier transforms to quadratic surfaces and decay of solutions of wave equations},
date={1977},
ISSN={0012-7094},
journal={Duke Math. J.},
volume={44},
number={3},
pages={705\ndash 714},
url={http://projecteuclid.org/euclid.dmj/1077312392},
review={\MR{512086}},
}

\bib{SS99}{book}{
author={Sulem, Catherine},
author={Sulem, Pierre-Louis},
title={The {N}onlinear Schr\"odinger {E}quation, {S}elf-focusing and wave collapse},
series={Applied Mathematical Sciences},
volume={139},
publisher={Springer-Verlag, New York},
date={1999},
pages={xvi+350},
isbn={0-387-98611-1},
review={\MR{1696311}},
}

\bib{Tsutsumi87}{article}{
author={Tsutsumi, Yoshio},
title={{$L^2$}-solutions for nonlinear {S}chr\"{o}dinger equations and nonlinear groups},
date={1987},
ISSN={0532-8721},
journal={Funkcial. Ekvac.},
volume={30},
number={1},
pages={115\ndash 125},
url={http://www.math.kobe-u.ac.jp/~fe/xml/mr0915266.xml},
review={\MR{915266}},
}

\bib{Y87}{article}{
author={Yajima, Kenji},
title={Existence of solutions for {S}chr\"{o}dinger evolution equations},
date={1987},
ISSN={0010-3616},
journal={Comm. Math. Phys.},
volume={110},
number={3},
pages={415\ndash 426},
url={http://projecteuclid.org/euclid.cmp/1104159313},
review={\MR{891945}},
}

\end{biblist}
\end{bibdiv}
\end{document}